\documentclass[11pt,reqno]{amsart}

\usepackage{amsmath,amssymb,amsthm,epsfig}
\usepackage{hyperref}
\usepackage{ulem}
\usepackage[mathscr]{euscript}
\usepackage[latin1]{inputenc}
\usepackage{xcolor}
\usepackage[nameinlink]{cleveref}
\usepackage{enumerate}
\usepackage{mathrsfs}
\usepackage[usenames,dvipsnames]{pstricks}
\usepackage{pst-grad} 
\usepackage{pst-plot} 

\hypersetup{
	colorlinks,
	linkcolor={blue!50!black},
	citecolor={blue!50!black},
	urlcolor={blue!50!black}
}

\newtheorem{theorem}{Theorem}
\newtheorem{definition}{Definition}
\newtheorem{lemma}{Lemma}
\newtheorem{proposition}{Proposition}
\newtheorem{corollary}{Corollary}
\newtheorem{remark}{Remark}

\numberwithin{equation}{section}
\numberwithin{theorem}{section}
\numberwithin{lemma}{section}
\numberwithin{corollary}{section}
\numberwithin{remark}{section} 
\numberwithin{proposition}{section}
\numberwithin{definition}{section}

\newcommand{\R}{\mathbb{R}}

\newcommand{\loc}{\operatorname{loc}}

\def\XXint#1#2#3{{\setbox0=\hbox{$#1{#2#3}{\int}$ }
\vcenter{\hbox{$#2#3$ }}\kern-.6\wd0}}

\title[A nonlocal dead-core problem]{Improved regularity for a nonlocal dead-core problem}

\author[D. Prazeres]{Disson dos Prazeres}
\address{Department of Mathematics, UFS, 49100-000, S\~ao Crist\'ov\~ao-SE, Brazil}{}
\email{disson@mat.ufs.br}

\author[R. Teymurazyan]{Rafayel Teymurazyan}
\address{Applied Mathematics and Computational Sciences (AMCS), Computer, Electrical and Mathematical Sciences and Engineering Division (CEMSE), King Abdullah University of Science and Technology (KAUST), Thuwal, 23955-6900, Kingdom of Saudi Arabia}{} 
\email{rafayel.teymurazyan@kaust.edu.sa} 

\author[J.M.~Urbano]{Jos\'{e} Miguel Urbano}
\address{Applied Mathematics and Computational Sciences (AMCS), Computer, Electrical and Mathematical Sciences and Engineering Division (CEMSE), King Abdullah University of Science and Technology (KAUST), Thuwal, 23955-6900, Kingdom of Saudi Arabia and CMUC, Department of Mathematics, University of Coimbra, 3000-143 Coimbra, Portugal}{} 
\email{miguel.urbano@kaust.edu.sa}

\begin{document}







\begin{abstract} 
We obtain improved regularity results for solutions to a nonlocal dead-core problem at branching points. Our approach, which does not rely on the maximum principle, introduces a new strategy for analyzing two-phase problems within the local framework, an area that remains largely unexplored.
\end{abstract}  

\date{\today}

\maketitle

\section{Introduction}

In this work, we obtain improved regularity properties for solutions of the \textit{two-phase} nonlocal problem
\begin{equation}\label{equation}
   -(-\Delta)^su=u_+^\gamma-u_-^\gamma \quad \textrm{in} \ B_1,
\end{equation}
where $(-\Delta)^s$ is the fractional Laplacian and $\gamma\in(0,1/3)$. Local versions of the \textit{one-phase} problem have attracted increasing attention in recent years (see, for example, \cite{ALT16, AT24, DT20, T16}) due to their wide range of applications, for instance, in optimizing resources in catalysis processes. In such reaction-diffusion models, the presence of a catalyst accelerates the rate of chemical reactions, leading to the formation of regions where the reactant concentration drops to zero and no reaction occurs. These regions are commonly known as \textit{dead-cores} in the literature, and placing catalysts there would be ineffective and result in resource wastage. The mathematical study of these phenomena has a long-standing history, dating back to the seminal work on the classical Alt-Phillips problem (cf. \cite{AP86, ATU24, P1, P2, WY22, Y13}).

The dead-core problem \eqref{equation} is nonlocal in nature, and the techniques used to treat the corresponding local models cannot be applied in our setting, requiring an alternative approach. For example, we cannot rely on the maximum principle to derive sharp regularity estimates as done in the one-phase local model of \cite{T16} corresponding to $s=1$ and $u\ge0$. The reason for this is two-fold. First, for the maximum principle to hold in the nonlocal setting, solutions would need to have a sign in the whole complement of the unit ball (see, for example, \cite[Theorem 6.1]{T24}). Second, while in one-phase models, points on the free boundary $\partial\{u>0\}$ correspond to local minima of the function, this property does not extend to the two-phase scenario.

Expanding on ideas from \cite{S15}, our strategy relies instead on the study of the growth of solutions according to the natural scaling of the equation, combined with an auxiliary Liouville-type result. This leads to improved sharp regularity results at branching points, where the function vanishes with some of its derivatives, typically up to order two. The approach comes with the caveat of restricting $\gamma\in(0,1/3)$, whereas in the \textit{one-phase} local case, we can take $\gamma\in(0,1)$. This is induced by the nonlocal nature of the problem, as the new approach requires a certain integrability condition (see the proof of \Cref{main result 1} below) only valid in this range. However, unlike the local result in \cite{T16}, this new scheme offers two key advantages. First, it extends to the two-phase case, regardless of whether the left-hand side of \eqref{equation} has a sign. In fact, the equation 
$$
(-\Delta)^su=u_+^\gamma
$$
has no direct local analogue as, for $s=1$, no dead-core phenomenon occurs, \textit{i.e.}, no non-negative solutions vanish in an interior region without being identically zero. The second significant advantage builds on the first, introducing a new strategy for analyzing the two-phase problem in the \textit{local} setting. Since our estimates remain uniform, we can establish sharp regularity results for solutions of the local two-phase problem by passing to the limit as $s \nearrow 1$ in \eqref{equation} (see \Cref{local two-phase problem}). The one-phase local problem, studied in \cite{T16} (see also \cite{AT24, DT20}), relies heavily on the maximum principle, which is why the corresponding two-phase problem remained unresolved. Our approach circumvents this difficulty, as it does not rely on the maximum principle or the sign of the right-hand side. 

To remain at least in the $C^{1,\sigma}-$regularity regime, we restrict our analysis to the range $s\in\left(1/2,1\right)$. Observe that as long as the right-hand side is bounded, we have, \textit{a priori}, that solutions of \eqref{equation} are locally of class $C^{1,\sigma}$, for any $0<\sigma<2s-1$ (see \Cref{CNtheorem} below). Since, in general, $u_+$ and $u_-$ are at most Lipschitz and the right-hand side in \eqref{equation} is then locally $C^\gamma$, the best regularity one can hope for solutions of \eqref{equation}, using Schauder theory, is $C^{2s+\gamma}_{\loc}$. However, our main result reveals that solutions are indeed $C^{\frac{2s}{1-\gamma}}$ at branching points. Note that
$$
1+\frac{2s-1+\gamma}{1-\gamma} = \boxed{\frac{2s}{1-\gamma}>2s+\gamma} = 1+\frac{2s-1+\gamma}{1},
$$
for any $\gamma \in (0,1)$ and $s\in \left(1/2,1\right)$, \textit{i.e.}, we obtain higher regularity than one could hope for relying on Schauder theory alone. More precisely, assuming $x_0$ is a branching point for $u$, \textit{i.e.}, $u(x_0)=|D^{\nu-1}u(x_0)|=|D^\nu u(x_0)|=0$, we show that around it one has precisely the growth given by
\begin{equation}\label{mainestimate}
|u(x)|\leq C|x-x_0|^{\frac{2s}{1-\gamma}}.
\end{equation}
Here $\nu$ is a constant defined by
\begin{equation*}
    \nu=\left\{\begin{array}{cc}
    1, &  s<1-\gamma,\\
    2, &  s>1-\frac{\gamma}{2},
    \end{array}
    \right.
\end{equation*}
The idea of the proof is to study the growth of the scaled function
$$
v_r(x):=\frac{u(rx)}{r^{\frac{2s}{1-\gamma}}},
$$
as $r\rightarrow 0$. Obtaining a Liouville-type theorem, we ensure that $v_r$ grows like a polynomial of degree $\nu$, which implies \eqref{mainestimate} by contradiction. 

In the nonlocal setting, a similar problem was studied in \cite{Y13}. However, in addition to the fact that the model considered therein does not generate dead-cores, the result heavily depends on the celebrated Caffarelli-Silvestre extension argument (\textit{cf.} \cite{CS07}), which limits its flexibility and prevents generalizations to a broader class of operators. On the contrary, our approach is flexible enough to be extended to fully nonlinear equations (\textit{cf.} \Cref{s6}).

\smallskip

The paper is organized as follows. In \Cref{s2}, we introduce the necessary notation and auxiliary results. The primary technical tool, a Liouville-type theorem, is presented in \Cref{s3}, while our main result on improved regularity at branching points is proved in \Cref{s4}. In \Cref{s5}, we explore some consequences of the main result, including its local implications and a Liouville-type property. In the final \Cref{s6}, we address the extension to the fully nonlinear case.

\medskip

\section{Notation and auxiliary results}\label{s2}

This section gathers some notation, the notion of solution, some remarks about existence, and a few auxiliary well-known results.

\subsection{Notation}
We denote with $B_r(x_0)$ the open ball of radius $r$ centered at $x_0$, and write $B_r:=B_r(0)$. For a multi-index $\beta=(\beta_1,\beta_2,\ldots,\beta_n)$, as usual, we put $|\beta|:=\beta_1+\beta_2+\ldots+\beta_n$. For $\alpha\in(0,1)$, the H\"older semi-norm is defined as 
$$
[u]_{C^{\alpha}}:=\sup_{x\neq y}\frac{|u(x)-u(y)|}{|x-y|^\alpha},
$$
$$
[u]_{C^{1+\alpha}}:=\max_{|\beta|=1} \left[ D^\beta u \right]_{C^{\alpha}},
$$
and
$$
[u]_{C^{2+\alpha}}:=\max_{|\beta|=2} \left[ D^\beta u \right]_{C^{\alpha}},
$$
where $D^\beta u:=\partial_{x_1}^{\beta_1}\ldots\partial_{x_n}^{\beta_n}u$.

The fractional Laplacian is the nonlocal operator defined by
\begin{equation*}	
(-\Delta)^su(x):=c_{n,s}\,\textrm{P.V.}\int_{\R^n}\frac{u(x)-u(y)}{|x-y|^{n+2s}}\,dy,
\end{equation*}
where $s \in (0,1)$ and $c_{n,s}$ is a normalization constant, depending only on $n$ and $s$. We use $u_+:=\max(u,0)$ and $u_-:=(-u)_+$. We also use the norm
$$
\|u\|_{L^1_s(\R^n)}:=\int_{\R^n}\frac{|u(y)|}{1+|y|^{n+2s}}dy.
$$

\subsection{Existence of solutions}

In the spirit of \cite{BCI08, BI08} (see also \cite{CLU14, CS09, CTU20}), and with the goal of extending our results to a broader class of fully nonlinear operators (see \Cref{s6}), solutions to \eqref{equation} are understood in the viscosity sense according to the following definition. 

\begin{definition}\label{d2.2}

A function $u:\R^n\rightarrow\R$ is called a viscosity subsolution (supersolution) of \eqref{equation}, and we write 
$$
-(-\Delta)^su\geq (\leq)\, u_+^\gamma-u_-^\gamma,
$$
if $u$ is upper (lower) semi-continuous in $\overline{B_1}$ and, whenever $x_0\in B_1$, $B_r(x_0)\subset B_1$, for some $r$, and $\varphi \in C^2(\overline{B_r(x_0)})$ satisfies
$$\varphi(x_0)=u(x_0) \quad {\rm and} \quad \varphi(y)> (<) \, u(y), \ \forall y\in B_r(x_0) \setminus \{x_0\},$$
then, if we let
$$
v:=
\left\{
\begin{array}{ll}
\varphi & \hbox{in } \ B_r(x_0)\\
\\
u & \hbox{in } \ \R^n\setminus B_r(x_0),
\end{array}
\right.
$$
we have $-(-\Delta)^sv(x_0)\geq (\leq) \, v_+^\gamma(x_0)-v_-^\gamma(x_0)$.  

A function is called a viscosity solution if it is both a viscosity subsolution and a viscosity supersolution.
\end{definition}

For the existence of viscosity solutions, we refer the reader to \cite[Theorem 1.2]{BCI08}. The rough idea is that, since the comparison principle holds for the fractional Laplacian (see \cite[Theorem 5.2]{CS09}, \cite[Theorem 2.5]{CTU20} or \cite[Corollary 6.1]{T24}), the classical Perron's method then leads to the existence of the solution to the Dirichlet problem (see \cite{CIL92}, for example). The comparison principle for $$Lu:=-(-\Delta)^su-u_+^\gamma+u_-^\gamma$$  follows from \cite[Theorem 3]{BI08}. Its proof relies on the nonlocal Jensen-Ishii lemma, established in \cite[Lemma 1]{BI08}, and makes use of the technique of jets with ideas from \cite{CIL92}. For the sake of completeness, we state it below.
\begin{theorem}\label{comparison}
  If $u_1,u_2\in C(\R^n)$, $Lu_1\le0\le Lu_2$ in $B_1$, and $u_1\ge u_2$ in $\R^n\setminus B_1$, then $u_1\ge u_2$ in $\R^n$.  
\end{theorem}

We conclude this section with regularity results for solutions of the homogeneous and the non-homogeneous fractional Laplace equation.
For the proof of the following theorem, we refer the reader to \cite[Theorem 27]{CS11} (see also \cite[Theorem 13.1]{CS09}).

\begin{theorem}\label{regularity}
Let $s\ge s_0>0$ and $u\in C(\overline{B}_1)$ be such that $\|u\|_{L^1_s(\R^n)}<\infty$. If $(-\Delta)^su=0$ in $B_1$, then there exists $\sigma>0$, depending only on $n$ and $s_0$, such that $u\in C^{1+\sigma}(B_{1/2})$. Moreover,
$$
\|u\|_{C^{1+\sigma}(B_{1/2})}\le C\left(\|u\|_{L^\infty(B_1)}+\|u\|_{L^1_s(\R^n)}\right),
$$
where $C>0$ is a constant, depending only on $n$ and $s_0$.
\end{theorem}

The proof of the following theorem can be found in \cite[Theorem 61]{CS11} (see also \cite[Theorem 52]{CS11}).

\begin{theorem}\label{CNtheorem}
If $s\in\left(1/2,1\right)$, and $u$ is a bounded solution of
$$
-(-\Delta)^su=f \quad \textrm{in} \ B_1,
$$
where $f\in L^\infty(B_1)$, then $u\in C^{1+\sigma}(B_{1/2})$, for any $\sigma<2s-1$. Moreover,
$$
\|u\|_{C^{1+\sigma}(B_{1/2})}\le C\left(\|u\|_{L^\infty(\R^n)}+\|f\|_{L^\infty(B_1)}\right),
$$
for a constant $C>0$, depending only on $n$.
\end{theorem}

Finally, we recall Schauder-type estimates for the fractional Laplacian (see, for example, \cite[Remark 9.1]{T24}).

\begin{theorem}\label{Schauder}
    If $s\in(0,1)$ and $(-\Delta)^su\in C^\sigma (B_1)\cap C(\overline{B}_1)$, for some $\sigma>0$, then $u\in C_{\loc}^{\sigma+2s}(B_1)$.
\end{theorem}

\medskip

\section{Liouville-type results}\label{s3}

This section and the next form the core of the paper. We prove Liouville-type results in the spirit of \cite{S15, HY} that will be used later to derive the sharp regularity for the dead-core problem.

\begin{theorem}\label{SeminormGrowth}
    Let $s\in(1/2,1)$ and
    \begin{equation}\label{Zeroquotient}
    (-\Delta)^s \left( u(x+h)-u(x) \right)=0, \quad x\in B_1,  
    \end{equation}    
    for all $h\in\R^n$. If, for $0\le\alpha<\beta<1$, there holds
    \begin{equation}\label{crescimento da norma}
    [u]_{C^{\nu+\alpha}(B_R)}\le CR^{\beta-\alpha},\quad \forall R\ge1,
    \end{equation}
    where $\nu=1,2$, and $C>0$ is a constant, depending only on $\alpha$ and $\beta$, then 
    $$
    u(x)=u(0)+Du(0) \cdot x+ \frac{\nu-1}{2} x^T \cdot D^2u(0) x, \quad x\in\R^n.
    $$
\end{theorem}

\begin{proof} 
We divide the proof into two steps, depending on the value of $\nu$.

\medskip

\textsc{Step 1.} Suppose $\nu=1$, and let $e\in\R^n$, with $|e|=1$. Set 
$$
w(x):=u(x)-u(0)-Du(0)\cdot x,
$$
and note that $w(0)=|Dw(0)|=0$, and $w$ satisfies \eqref{Zeroquotient}-\eqref{crescimento da norma}. Define $w_e:= Dw\cdot e$. Since $w_e(0)=0$, \eqref{crescimento da norma} yields
\begin{equation}\label{gradestimateoudsite}
    \begin{split}
    |w_e(x)|&= |w_e(x)-w_e(0)|\leq [w]_{C^{1,\alpha}(B_{|x|})}|x|^{\alpha}\\
    &\le C |x|^{\beta-\alpha} |x|^{\alpha}=C|x|^\beta,
    \end{split}
\end{equation}
for all $x \notin B_1$. Hence, since $\beta<1<2s$,
$$
\int_{B_1^c}\frac{|w_e(y)|}{|y|^{n+2s}}\,dy \leq C \int_{B_1^c}\frac{|y|^\beta}{|y|^{n+2s}}\,dy = C \int_{B_1^c}\frac{1}{|y|^{n+2s-\beta}}\,dy<\infty.
$$
On the other hand, \eqref{crescimento da norma} with $R=1$ gives
\begin{equation}\label{gradestimateinside}
    |w_e(x)|\le C, \quad x \in  B_1.
\end{equation}
By stability (see \cite[Lemma 5]{CS11}, and also \cite[Lemma 4.5 and Corollary 4.7]{CS09}), \eqref{Zeroquotient} then implies 
$$
(-\Delta)^s w_e=0 \quad \textrm{in} \ B_1.
$$
By \Cref{regularity}, there exists a constant $\sigma>0$, depending only on $n$, such that
$$
\|w_e\|_{C^{1,\sigma}\left( B_{1/2} \right)}<C_1,
$$
for a constant $C_1>0$, depending only on $n$ and $\beta$. 

Now, for $\rho>0$, set 
$$
v(x):=\rho^{-\beta} w_e(\rho x)
$$
and, recalling \eqref{gradestimateoudsite}, note that 
$$
|v(x)|=\rho^{-\beta}|w_e(\rho x)|\le C|x|^\beta, \quad\mbox{for all }\ x \notin B_{1/\rho}.
$$
Also, using \eqref{gradestimateinside}, in $B_{1/\rho}$ we have
$$
|v(x)|=\rho^{-\beta}|w_e(\rho x)|\le C\rho^{-\beta}<C,
$$
for a $\rho>0$ large enough. Thus,
$$
\|v\|_{L^\infty(B_1)}<C.
$$
Hence, again by \Cref{regularity}, there exists $\sigma_*>0$ such that 
$$
\|v\|_{C^{1,\sigma_*}(B_{1/2})}\le C_2 \left(\|v\|_{L^{\infty}(B_1)}+\|v\|_{L^1_{s}(\R^n)}\right),
$$
where $C_2>0$ is a constant, depending only on $n$. Since
$$
\| Dv\|_{L^{\infty}(B_{1/2})}=\rho^{1-\beta}\| Dw_e\|_{L^{\infty}(B_{1/2\rho})},
$$
if follows that  
\begin{equation}\label{important estimate}
    \| Dw_e\|_{L^{\infty}(B_{1/2\rho})}\le\rho^{\beta-1}C_3,
\end{equation}
where $C_3>0$ is a constant, depending only on $n$, $\beta$ and $C>0$ from \eqref{crescimento da norma}. Letting $\rho\to\infty$ in \eqref{important estimate}, one gets $Dw_e(0)=0$, and since the fractional Laplacian is translation invariant, we deduce that $Dw_e\equiv 0$, for any $e\in\R^n$, $|e|=1$. Hence, $w_e$ is a constant, and as $w_e(0)=0$, $w_e$ must be identically zero, implying that $w$ is a constant. But $w(0)=0$, therefore $w\equiv0$, and hence, 
$$
u(x)=u(0)+Du(0)\cdot x.
$$

\medskip

\textsc{Step 2.} If $\nu=2$, set 
$$
w(x):=u(x)-u(0)-Du(0)\cdot x-\frac{1}{2}x^T\cdot D^2u(0)x,
$$
and observe that it satisfies \eqref{Zeroquotient}-\eqref{crescimento da norma}, while $w(0)=|w_{e}(0)|=|w_{ee}(0)|=0$, where $e\in\R^n$ is a unit vector. Arguing as above, we arrive at $Dw_{ee}=0$ for any unit vector $e\in\R^n$, meaning that $w_{ee}$ is a constant, \textit{i.e.}, $w_{ee}\equiv0$. The latter yields $w_e=0$, which then gives $w\equiv0$. Thus, 
$$
u(x)=u(0)+Du(0)\cdot x+\frac{1}{2}x^T\cdot D^2u(0)x.
$$
\end{proof}

The following result is a simple consequence of the mean value theorem.

\begin{lemma}\label{caracterizacao de c1alpha}
Let $u\in C^{\nu,\alpha}(\R^n)$, where $\nu=1,2$ and $\alpha\in(0,1)$. If 
$$
u(0)=|D^{\nu-1}u(0)|=|D^\nu u(0)|=0
$$
and
$$
\sup_{0<r<1/2}r^{\alpha-\beta}[u]_{C^{\nu+\alpha}(B_r)}\le A,
$$
for $\beta>\alpha$ and a constant $A>0$, then
$$
|u(x)|\leq A|x|^{\nu+\beta}, \quad x\in B_{1/2}.
$$
\end{lemma}

\begin{proof}
As $u(0)=0$, by the mean value theorem, for $x\in B_{1/2}$, one has
\begin{equation}\label{ineq1}
   |u(x)|=|u(x)-u(0)|\le|Du(\xi_1)||x|,
\end{equation}
for some $\xi_1\in B_{|x|}$. 

For $\nu=1$, since $Du(0)=0$ and $Du\in C^{\alpha}\left(B_{1/2}\right)$, we have
$$
\frac{|Du(\xi_1)||x|}{|x|^{1+\beta}}=\frac{|Du(\xi_1)-Du(0)||x|}{|x|^{1+\beta}}\le|x|^{\alpha-\beta}\left[u\right]_{C^{1+\alpha} \left( B_{|x|}\right)}\le A,
$$
which combined with \eqref{ineq1} concludes the proof in this case. 

Similarly, if $\nu=2$, since $|Du(0)|=|D^2u(0)|=0$, employing the mean value theorem once more, for some $\xi_2\in B_{|\xi_1|}$, we estimate
\begin{equation*}
    \begin{split}
        \frac{|Du(\xi_1)||x|}{|x|^{2+\beta}}&=\frac{|Du(\xi_1)-Du(0)||x|}{|x|^{2+\beta}}\\
        &\le\frac{|D^2u(\xi_2)-D^2u(0)||x|^2}{|x|^{2+\beta}}\\
        &\le|x|^{\alpha-\beta}\left[u\right]_{C^{2+\alpha}\left( B_{|x|}\right)}\\
        &\le A,
    \end{split}
\end{equation*}
which concludes the proof also for $\nu=2$, thanks to \eqref{ineq1}.
\end{proof}

\medskip

\section{Improved regularity at branching points}\label{s4}
In this section, we prove the main result of this paper. As observed earlier, unlike the local one-phase problem treated in \cite{T16}, we cannot rely on the maximum principle to derive sharp regularity estimates. We start with the precise definition of a branching point.

\begin{definition}\label{branching points}
A point $x_0 \in B_1$ is called a branching point for the function $u:B_1 \to \R$ if 
$$u(x_0)=|Du(x_0)|=|D^2u(x_0)|=0.$$
\end{definition}

\begin{theorem}\label{main result 1}
    Let $\gamma\in(0,1/3)$ and $s>1-\frac{\gamma}{2}$. If $u\in L^\infty(\R^n)$ is a viscosity solution of \eqref{equation} and $x_0\in B_{1/2}$ is its branching point, then there exists a constant $C>0$, depending only on $n$ and $\gamma$, such that
\begin{equation}\label{growth estimate 2025}  
    |u(x)|\le C\|u\|_{L^\infty(\R^n)}|x-x_0|^{\frac{2s}{1-\gamma}}, \quad \forall x \in B_{1/2}(x_0).      
\end{equation} 
As a consequence, $u\in C^{\frac{2s}{1-\gamma}}$ at branching points.
\end{theorem}

\begin{proof}
By \Cref{CNtheorem} and \Cref{Schauder}, we have $u\in C^{2s+\gamma}(B_{1/2})$. Moreover,
$$
\|u\|_{C^{2s+\gamma}(B_{1/2})}\le C\left(\|u\|_{L^\infty(\R^n)}+\|u\|_{L^\infty(B_1)}^\gamma\right),
$$
for a constant $C>0$, depending only on $n$ and $s$. With no loss of generality, we may assume $x_0=0$ and $\|u\|_{L^\infty(\R^n)}=1$,
and we need to show that
\begin{equation}\label{nu=1}
    |u(x)|\le C|x|^\frac{2s}{1-\gamma}.
\end{equation}
We argue by contradiction and assume that \eqref{nu=1} fails. Then, there exist sequences $u_k$ of solutions of \eqref{equation} and points $x_k$, such that $0$ is a branching point of $u_k$, and
$$
[u_k]_{C^{2s+\gamma}(B_{1/2})}\le2C,
$$
with
\begin{equation}\label{contradictory assumption}
    |u_k(x_k)|>k|x_k|^{\frac{2s}{1-\gamma}}.
\end{equation}
Set 
\begin{equation}\label{definitiontheta}
    \theta_k(r^{\prime} ):=\sup_{r^{\prime} <r<1/2}r^{2s+\gamma-\frac{2s}{1-\gamma}}\left[u_k\right]_{C^{2s+\gamma}(B_r)}.
\end{equation}
Observe that $2s+\gamma-\frac{2s}{1-\gamma}<0$ and
$$
\lim_{r^{\prime}\rightarrow 0}\theta_k(r^{\prime} )=\sup_{0<r<1/2}r^{2s+\gamma-\frac{2s}{1-\gamma}}\left[u_k\right]_{C^{2s+\gamma}(B_r)}.
$$
Then, \eqref{contradictory assumption} and \Cref{caracterizacao de c1alpha} yield
$$
\lim_{r^{\prime}\rightarrow 0}\theta_k(r^{\prime})>k.
$$
Thus, there exists a sequence $r_k>\frac{1}{k}$, such that
\begin{equation}\label{contradicao2}
r_k^{2s+\gamma-\frac{2s}{1-\gamma}}\left[u_k\right]_{C^{2s+\gamma}(B_{r_k})}\ge\frac{1}{2}\theta_k(1/k)\ge\frac{1}{2}\theta_k(r_k). 
\end{equation}
From the first inequality in \eqref{contradicao2}, we conclude that 
$$
r_k \longrightarrow 0,\quad \textrm{as}\quad k\to\infty.
$$
Set now 
\begin{equation}\label{definitionvk}
    v_k(x_k):=\frac{u_k(r_k x_k)}{\theta_k(r_k)r_k^{\frac{2s}{1-\gamma}}}
\end{equation}
and note that, for $1\le R\le\frac{1}{2r_k}$, one has
\begin{equation}\label{vkestimate}
\begin{split}
    \left[v_k\right]_{C^{2s+\gamma}(B_R)}&=\frac{1}{\theta_k(r_k)r_k^{\frac{2s}{1-\gamma}}}\left[u_k\right]_{C^{2s+\gamma}(B_{r_kR})}r_k^{2s+\gamma}\\
    &=\frac{(r_kR)^{2s+\gamma-\frac{2s}{1-\gamma}}}{\theta_k(r_k)}\left[u_k\right]_{C^{2s+\gamma}(B_{r_k R})}R^{\frac{2s}{1-\gamma}-2s-\gamma}.
    \end{split}
\end{equation}
Since
$$
(r_k R)^{2s+\gamma-\frac{2s}{1-\gamma}}\left[u_k\right]_{C^{2s+\gamma}(B_{r_k R})}\le\theta_k(r_k R)\le\theta_k(r_k),
$$
employing \eqref{vkestimate}, we get
\begin{equation}\label{seminorm estimate on vk}
    [D^2v_k]_{C^{2s+\gamma-2}(B_R)}=\left[v_k\right]_{C^{2s+\gamma}(B_R)}\le R^{\frac{2s}{1-\gamma}-2s-\gamma},\quad\forall R\ge1.
\end{equation}
Since $0$ is a branching point for $v_k$, using the mean value theorem and \eqref{seminorm estimate on vk}, one has
\begin{equation*}
    \begin{split}
        |v_k(x)|&\le |Dv_k(\xi)||x|\le |D^2v_k(\xi')||x|^2\\
        &\le[D^2v_k]_{C^{2s+\gamma-2}(B_R)}|x|^{2s+\gamma}\\
        &\le R^{\frac{2s}{1-\gamma}-2s-\gamma}|x|^{2s+\gamma},
    \end{split}
\end{equation*}
where $\xi$ is a point on the line segment connecting the origin to $x$, and $\xi'$ is a point on the line segment connecting the origin to $\xi$. Now, if $\eta$ is a smooth function such that $\eta \equiv 1$ in $B_{1/2}$ and $\eta = 0$ outside $B_1$, then for any unit vector $e$, one has
$$
\int_{B_1}\eta\cdot D_{ee} v_k\,dx=\int_{B_1}D_{ee}\eta\cdot v_k\,dx\leq C_n,
$$
where $C_n>0$ is a constant, depending only on $n$. Therefore, there exists $z\in B_1$ such that $|D^2v_k(z)|\leq C_n$, and \eqref{seminorm estimate on vk} implies 
$$
|D_{ee}v_k(x)-D_{ee}v_k(z)|\le R^{\frac{2s}{1-\gamma}-2s-\gamma}|x-z|^{2s+\gamma-2}.
$$
Hence, for $x\in B_R$ and $1\leq R \le\frac{1}{2r_k}$, we get
$$
|D_{ee}v_k(x)|\le C_n+ R^{\frac{2s}{1-\gamma}-2s-\gamma}|x-z|^{2s+\gamma-2}\le CR^{\frac{2s}{1-\gamma}-2},
$$
for a constant $C>0$, depending only on $n$. Thus, up to a subsequence, $v_k$ converges to some $v_0$ in $C^{2s+\gamma}_{\loc}(B_R)$, as $k\to\infty$, and thanks to \eqref{seminorm estimate on vk},
\begin{equation}\label{v0estimate}
\left[v_0\right]_{C^{2s+\gamma}(B_R)}\le R^{\frac{2s}{1-\gamma}-2s-\gamma}, \quad \forall R\ge1.
\end{equation}
Additionally, from the second inequality in \eqref{contradicao2}, we deduce that
\begin{equation}\label{limitseminormestimate2}
     \left[v_0\right]_{C^{2s+\gamma}(B_{1})}\ge\frac{1}{2}.
\end{equation}
Observe that, for a fixed $h\in\R^n$, using the mean value theorem and \eqref{seminorm estimate on vk}, for  $|x|\ge1$, we have
\begin{equation}\label{vkgrowth2}
    \begin{split}
    |v_k(x+h)-v_k(x)|&\le|Dv_k(\eta)||h|\\
    &\le|D^2v_k(\eta')|(|x|+|h|)|h|\\
    &\le[v_k]_{C^{2s+\gamma}(B_{|x|+|h|})}(|x|+|h|)^{2s+\gamma-1}|h|\\
     &\le (|x|+|h|)^{\frac{2s}{1-\gamma}-1}|h|\\
     &\le C_h|x|^{\frac{2s}{1-\gamma}-1},
    \end{split}
\end{equation}
where $C_h>0$ is a constant, depending only on $h$ and $\frac{2s}{1-\gamma}$. Here, $\eta$ is a point on the line segment connecting $x$ to $x+h$, and $\eta'$ is a point on the line segment connecting the origin to $h$. Since $\gamma\in(0,1/3)$, then
$$
\gamma<\frac{1}{1+2s},
$$
which guarantees that the right-hand side of \eqref{vkgrowth2} is in $L^1_s(\R^n)$. Furthermore, as $u_k$ solves \eqref{equation}, a direct calculation reveals that
\begin{equation}\label{becka}
\begin{split}
    &-(-\Delta)^s\left(v_k(x+h)-v_k(x)\right)\\
    &=\frac{1}{\theta_k(r_k)r_k^{\frac{2s\gamma}{1-\gamma}}}\left[\left(u_k(x+h)\right)_+^\gamma-\left(u_k(x)\right)_+^\gamma-\left(u_k(x+h)\right)_-^\gamma+\left(u_k(x)\right)_-^\gamma\right].
\end{split}
\end{equation}
We can then pass to the limit, as $k\to\infty$, in \eqref{becka} (see \cite[Lemma 5]{CS11}, and also \cite[Lemma 4.5 and Corollary 4.7]{CS09}) and arrive at
\begin{equation*}\label{limit equation}
    (-\Delta)^s(v_0(x+h)-v_0(x))=0\,\,\mbox{ in }\,\,\R^n.
\end{equation*}
Note that since $0$ is a branching point for $u_k$, \eqref{definitionvk} implies that it is a branching point for $v_0$. This fact, combined with \Cref{SeminormGrowth} applied to $v_0$, with 
    $$
    \nu=2, \quad \beta:=\frac{2s}{1-\gamma}-2<1 \quad \text{and} \quad \alpha:=2s+\gamma-2<1,
    $$ 
    implies $v_0\equiv0$, which contradicts \eqref{limitseminormestimate2}.
\end{proof}

\begin{corollary}\label{c4.1}
     Under the conditions of \Cref{main result 1}, there exists a constant $C>0$, depending only on $n$ and $\gamma$, such that for any $r<1/2$ one has
    $$
     \sup_{B_r(x_0)}|u|\leq  C\|u\|_{L^\infty(\R^n)}r^{\frac{2s}{1-\gamma}}.
    $$   
\end{corollary}
\begin{remark}
Note that if $s\in(1/2,1-\gamma)$, with $\gamma\in(0,1/3)$, and $u\in L^\infty(\R^n)$ is a viscosity solution of \eqref{equation}, then the conclusion of \Cref{main result 1} still holds, provided $x_0\in B_{1/2}$ is such that $u(x_0)=|Du(x_0)|=0$. The proof is similar to that of \Cref{main result 1}; at the end, to get a contradiction, one needs to apply \Cref{SeminormGrowth} to $v_0$ with
$$
\nu=1,\quad \beta:=\frac{2s}{1-\gamma}-1<1\quad \text{and}\quad \alpha:=2s+\gamma-1<1.
$$
\end{remark}

\begin{remark}
    The sharp regularity exponent provided by Theorem \ref{main result 1} plays a crucial role in understanding the behaviour of the solution in the smallness regime, which, in turn, enables the analysis of the regularity of the free boundary near branching points (see, for example, \cite{ATU24}).
\end{remark}

\medskip

\section{Consequences and beyond}\label{s5}

In this section, we prove three consequences of our main result. We start by observing that the approach to prove \Cref{main result 1} introduces a new strategy for studying the regularity of solutions for two-phase problems in the \textit{local} framework. 

\begin{theorem}\label{local two-phase problem}
    Let $u$ be a viscosity solution of 
    \begin{equation}\label{local equation}
    \Delta u=u_+^\gamma-u_-^\gamma\quad \textrm{in}\ B_1,
    \end{equation}
    where $\gamma\in(0,1/3)$. If $x_0\in B_{1/2}$ is a branching point for $u$, then there exists a constant $C>0$, depending only on $n$ and $\gamma$, such that
    $$
    |u(x)|\le C\|u\|_{L^\infty(\R^n)}|x-x_0|^{\frac{2}{1-\gamma}}, \quad \forall x \in B_{1/2}(x_0).
    $$
    As a consequence, $u\in C^{\frac{2}{1-\gamma}}$ at branching points.    
\end{theorem}

\begin{proof}
Since 
$$
\lim_{s\nearrow 1}(-\Delta)^su=-\Delta u
$$
(see, for example, \cite[Lemma 4.1]{T24}), using \Cref{main result 1} and passing to the limit in \eqref{growth estimate 2025} as $s\nearrow 1$, we conclude the desired result.
\end{proof}

\begin{remark}
    Since all points in $\{x_n=0\}$ are branching points for
\[
u(x):=\left\{
\begin{array}{rcl}
x_n^{\frac{2}{1-\gamma}} & , & x_n\geq 0 \\
 -(-x_n)^{\frac{2}{1-\gamma}}& , & x_n<0,
\end{array}
\right.
\]
the improved regularity result of \Cref{local two-phase problem} is optimal.
\end{remark}

\begin{proposition}\label{one-phase branching points}
    If $u$ is a viscosity solution of 
    \begin{equation}\label{one-phase local}
        \Delta u=u_+^\gamma \quad \textrm{in} \ B_1,
    \end{equation}
    then all the points on $\partial\{u>0\}$ are branching points for $u$.
\end{proposition}

\begin{proof}
    Indeed, if $x_0\in\partial\{u>0\}$, then $u(x_0)=|Du(x_0)|=0$, where the last equality follows from the fact that $u$ takes its minimum at $x_0$. For the same reason, $D^2u(x_0)$ is non-negative definite, \textit{i.e.}, all the eigenvalues of $D^2u(x_0)$ are non-negative. On the other hand, if $\lambda_i\ge0$ are the eigenvalues of $D^2u(x_0)$, one has
    $$    0=u_+^\gamma(x_0)=\Delta u(x_0)=\textrm{trace}(D^2u(x_0))=\sum_{i=1}^n\lambda_i,
    $$
    therefore, $\lambda_i$=0, for all $i=1,2,\ldots,n$. Hence, $D^2u(x_0)=0$.
\end{proof}

\begin{remark}\label{remark on one-phase problem}
Solutions of \eqref{one-phase local} are $C^{\frac{2}{1-\gamma}}$ at the free boundary $\partial\{u>0\}$ (see \cite[Theorem 2]{T16}). In view of \Cref{one-phase branching points}, \Cref{local two-phase problem} generalizes this result to the two-phase case. Note, however, that while for \eqref{one-phase local} the range of $\gamma$ is $(0,1)$, for \eqref{local two-phase problem} $\gamma$ ranges in $(0,1/3)$.
\end{remark}

We now prove the second consequence of our main result.

\begin{theorem}
Under the conditions of \Cref{main result 1}, there exists $C>0$, depending only on $n$ and $\gamma$, such that
\begin{equation}\label{gradientestimate}
    \sup_{B_r(x_0)}|Du|\le C\|u\|_{L^\infty(\R^n)}r^{\frac{2s}{1-\gamma}-1},
\end{equation}
for $r<1/2$, where $x_0\in B_{1/2}$ is a branching point.
\end{theorem}

\begin{proof}    
    Without loss of generality, let $x_0=0$ and $\|u\|_{L^\infty(\R^n)}=1$. Set 
    $$
    \beta:=\frac{2s}{1-\gamma}-1.
    $$
    We argue by contradiction, assuming \eqref{gradientestimate} fails. Then, for each $k\in\mathbb{N}$, there exists $u_k$, a normalized solution of \eqref{equation}, such that, for some $r<1/2$, one has
    $$
    \sup_{B_{r}}|Du_k|>k\,r^\beta.
    $$    
    Let $k_r\in\mathbb{N}$ be such that
    $$
    2^{-(k_r+1)}\le r<2^{-k_r}.
    $$
    Then
    $$
    \sup_{B_{2^{-k_r}}}|Du_k|\ge\sup_{B_r}|Du_k|>k\,r^\beta\ge k2^{-(k_r+1)\beta},
    $$
    \textit{i.e.},    
    \begin{equation}\label{contradictoryassumption}
        \mu_{k_r}>k2^{-(k_r+1)\beta}, 
    \end{equation}
    where
    $$
    \mu_{k_r}:=\sup_{B_{2^{-k_r}}}|Du_k|.
    $$    
    By \Cref{c4.1}, there exists a constant $C>0$, depending only on $n$ and $\gamma$, such that  
    \begin{equation}\label{supestimate}      
    \sup_{B_1}|u_k(2^{-k_r} x)|=\sup_{B_{2^{-k_r}}}|u_k(x)|\le C2^{-k_r(\beta+1)}.
    \end{equation}   
    Set      
    $$
    v_k(x):=\frac{u_k(2^{-k_r}x)}{2^{-k_r}\mu_{k_r}}, \quad x\in B_{1}.
    $$
    Employing \eqref{supestimate} and \eqref{contradictoryassumption}, we estimate
    \begin{equation}\label{vk bound}
    |v_k(x)|\le C\frac{2^{-k_r(\beta+1)}}{2^{-k_r}k2^{-(k_r+1)\beta}}=C\frac{2^\beta}{k}.
    \end{equation}    
    Note that
    \begin{equation}\label{sup of grad vk}      \sup_{B_{1}}|Dv_k|=\sup_{B_{2^{-k_r}}}\frac{|Du_k(x)|}{\mu_{k_r}}=1.
    \end{equation}    
   Furthermore, from \eqref{contradictoryassumption} and \eqref{vk bound}, we get
   \begin{align}\label{fractional Laplacian estimate}
    |(-\Delta)^{s}v_k(x)|&=\mu_{k_r}^{-1}2^{(1-2s)k_r}|(-\Delta)^s u_k(2^{-k_r}x)|\nonumber\\
    &=\mu_{k_r}^{-1}2^{(1-2s)k_r}\left|(u_k)_+^{\gamma}(2^{-k_r}x)-(u_k)_-^\gamma(2^{-k_r}x)\right|\nonumber\\
    &=\mu_{k_r}^{\gamma-1}2^{(1-2s-\gamma)k_r}\left|(v_k)_+^{\gamma}(x)-(v_k)_-^\gamma(x)\right|\nonumber\\
    &<\left[k2^{-(k_r+1)\beta}\right]^{\gamma-1}2^{(1-2s-\gamma)k_r}\left|(v_k)_+^{\gamma}(x)-(v_k)_-^\gamma(x)\right|\nonumber\\
    &=k^{\gamma-1}2^{2s+\gamma-1}\left|(v_k)_+^{\gamma}(x)-(v_k)_-^\gamma(x)\right|\nonumber\\
    &\le\frac{C^{\gamma}2^{2s+\beta\gamma+\gamma}}{k}.
    \end{align}    
    Additionally, for any $0<\alpha<\sigma<2s-1$ and $2^{k_r-1}\ge R\ge1$, recalling \eqref{contradictoryassumption} and \Cref{CNtheorem}, one has
    \begin{align}\label{vk seminorm bound}
        [v_k]_{C^{1+\alpha}(B_R)}&=\frac{2^{-k_r(1+\alpha)}}{\mu_{k_r}}[u_k]_{C^{1+\alpha}(B_{2^{-k_r}R})}\nonumber\\
        &\le\frac{2^{-k_r(1+\alpha)}}{\mu_{k_r}}(2^{-k_r+1}R)^{\sigma-\alpha}[u_k]_{C^{1+\sigma}(B_{2^{-k_r}R})}\nonumber\\   
        &<\frac{2^{-k_r(1+\alpha)+(-k_r+1)(\sigma-\alpha)}}{k2^{-(k_r+1)\beta}}R^{\sigma-\alpha}[u_k]_{C^{1+\sigma}(B_{2^{-k}R})}\nonumber\\
        &=\frac{2^{\sigma-\alpha+\beta}}{k}2^{k_r\left(\frac{2s}{1-\gamma}-2-\sigma\right)}R^{\sigma-\alpha}[u_k]_{C^{1+\sigma}(B_{2^{-k}R})}\nonumber\\
        &\le\tilde{C}R^{\sigma-\alpha},
    \end{align}
    for a constant $\tilde{C}>0$, depending only on $\gamma$ and $s$. Here, the last inequality follows by choosing $\sigma>0$ sufficiently close to $2s-1$ so
    $$
    \frac{2s}{1-\gamma}-2-\sigma<0.
    $$
    Thus, up to a subsequence, $v_k$ converges to some $v_0$ in $C^{1+\alpha}(B_R)$, as $k\to\infty$. Moreover, thanks to \eqref{fractional Laplacian estimate} and \eqref{vk seminorm bound},
    $$
    (-\Delta)^sv_0=0
    $$
    and
    $$
    [v_0]_{C^{1+\alpha}(B_R)}\le\tilde{C}R^{\sigma-\alpha}.
    $$
    \Cref{SeminormGrowth} applies to $v_0$, yielding $v_0\equiv0$, which is a contradiction, since from \eqref{sup of grad vk}, one has
    $$
    \sup_{B_{1}}|Dv_0|=1.
    $$
\end{proof}
We finish this section by proving a Liouville-type result for solutions of \eqref{equation} with a certain growth at infinity.

\begin{theorem}
    Let $u\in L^\infty(\R^n)$ be a viscosity solution of 
    \begin{equation*}\label{entiresolution}
   -(-\Delta)^su=u_+^\gamma-u_-^\gamma \quad \textrm{in} \ \R^n,
\end{equation*}    
    with $\gamma\in\left(0,1/3\right)$ and $s\in\left(1/2,1\right)$. If $x_0\in\partial\{u>0\}\cap\left\{|Du|=0\right\}$, and 
    \begin{equation}\label{5.0}
        u(x)=o\left(|x-x_0|^{\frac{2s}{1-\gamma}}\right),\quad \textrm{as}\quad |x|\rightarrow \infty
    \end{equation}
    then $u\equiv0$.
\end{theorem}

\begin{proof}
    Without loss of generality, we may assume $x_0=0$. With $R>1$, set
    $$
    v_R(x):=\frac{u(Rx)}{R^\frac{2s}{1-\gamma}}.
    $$    
    Since $v_R(0)=0$, and $v_R$ is a viscosity solution of \eqref{equation}, \Cref{main result 1} gives
    \begin{equation}\label{5.3}
        |v_R(x)|\le C\|v_R\|_{L^\infty(\R^n)}|x|^{\frac{2s}{1-\gamma}}.
    \end{equation}    
    Observe that if $|Rx|$ is bounded, then $u(Rx)$ is bounded. Therefore, 
    \begin{equation}\label{vtozero}
        v_R \longrightarrow 0,\quad \textrm{as}\quad R\to\infty.
    \end{equation}
    In fact, \eqref{vtozero} remains true also when $|Rx|\to\infty$, as $R\to\infty$. Indeed, using \eqref{5.0}, for any fixed $x\neq0$, one gets
    $$
    v_R(x)=\frac{u(Rx)}{|Rx|^{\frac{2s}{1-\gamma}}}|x|^{\frac{2s}{1-\gamma}} \longrightarrow 0,\quad \textrm{as}\ R\to\infty.
    $$    
    We aim to show that $u\equiv0$. Suppose this is not the case, and there is a point $y\in\R^n$ such that $|u(y)|>0$. By choosing $R>0$ large enough so that $y\in B_R$, and using \eqref{5.3} and \eqref{vtozero}, we estimate
    $$
    \frac{|u(y)|}{|y|^\frac{2s}{1-\gamma}}\le\sup_{B_R}\frac{|u(x)|}{|x|^\frac{2s}{1-\gamma}}=\sup_{B_1}\frac{|v_R(x)|}{|x|^\frac{2s}{1-\gamma}}<\frac{|u(y)|}{2|y|^\frac{2s}{1-\gamma}},
    $$    
    reaching a contradiction.    
\end{proof}

\medskip

\section{Fully nonlinear case}\label{s6}

As observed earlier, our main result can be generalized to the fully nonlinear setting. More precisely, let 
\begin{equation}\label{equationfully}
F(D^{2s}u)=u_+^\gamma-u_-^\gamma\quad \textrm{in}\ B_1,
\end{equation}
where $D^{2s}u(x)$ is the matrix with $(i,j)-$entry
$$
\int_{\R^n}\delta u(x,y)\frac{\langle e_i,y\rangle \langle e_j,y\rangle}{|y|^{n+2s+2}}\,dy,
$$
where $\delta u(x,y):=u(x+y)+u(x-y)-2u(x)$ is the symmetric difference, and $\{e_i\}$ is the standard orthonormal basis of $\R^n$. In \eqref{equationfully}, the operator $F:\operatorname{Sym}(n)\rightarrow \R$ is assumed to be a uniformly elliptic operator that vanishes at the origin, \textit{i.e.}, $F(0)=0$, and 
\begin{equation}\label{ellipticity}
    \lambda\|N\|\leq F(M+N)-F(M)\leq\Lambda\|N\|,
\end{equation}
for some constants $0<\lambda\le\Lambda$, and for any $M,N\in \text{Sym}(n)$ with $N\geq 0$. Furthermore, we suppose that $F(M)$ is differentiable with respect to $M$, and     \begin{equation}\label{modulusofcontinuity}
    \|DF(M)-DF(N)\|\leq\omega(\|M-N\|),
\end{equation}
for a modulus of continuity $\omega$, and for all $M,N\in \operatorname{Sym}(n)$.
The following result is from \cite[Proposition 2.2]{HY} (see also \cite{ASS, PT}). It unlocks the proof of the corresponding improved regularity result for \eqref{equationfully}.

\begin{proposition}\label{convergence}
    If $F_k:\operatorname{Sym}(n)\to\R$ is a sequence of operators vanishing at the origin and satisfying \eqref{ellipticity}-\eqref{modulusofcontinuity} with the same ellipticity constants, then, up to a subsequence,
    $$
    \rho^{-1}F_k(\delta M) \longrightarrow DF_0(0)M,
    $$
    locally uniformly, as $\rho\to0$, where $F_0$ is an operator satisfying the same conditions.
\end{proposition}

The following result generalizes \Cref{main result 1}, and its proof is based on similar arguments. We sketch it here for the reader's convenience.

\begin{theorem}
    Let $F$ satisfy \eqref{ellipticity}-\eqref{modulusofcontinuity} and $F(0)=0$. If $u\in L^\infty(\R^n)$ is a viscosity solution of \eqref{equationfully}, for $\gamma\in(0,1/3)$ and $s>1-\frac{\gamma}{2}$, and $x_0\in B_{1/2}$ is a branching point of $u$, then there exists a constant $C>0$, depending only on $n$, $\gamma$, $\lambda$ and $\Lambda$, such that
    $$
    |u(x)|\le C\|u\|_{L^\infty(\R^n)}|x-x_0|^{\frac{2s}{1-\gamma}}, \quad \forall x \in B_{1/2}(x_0). 
    $$    
\end{theorem}

\begin{proof}
    Without loss of generality, we assume $x_0=0$ and $\|u\|_{L^\infty(\R^n)}=1$. We make use of \Cref{convergence} and argue as in the proof of \Cref{main result 1}. More precisely, if the conclusion fails, then there exist sequences $u_k$, $x_k$, and $F_k$ such that $0$ is a branching point of $u_k$,
    $$
    F_k( D^{2s}u_k)=(u_k)_+^\gamma-(u_k)_-^\gamma,
    $$
$$
[u_k]_{C^{2s+\gamma}(B_{1/2})}\le2C,
$$
but
\begin{equation*}
    |u_k(x_k)|>k|x_k|^{\frac{2s}{1-\gamma}}.
\end{equation*}
Let now $\theta_k$ be as in \eqref{definitiontheta}. As observed in the proof of \Cref{main result 1}, there exists $r_k\to0$, such that \eqref{contradicao2} holds. A direct calculation shows that
\begin{equation}\label{fullynonlinearequation}
    \begin{split}
        &\tilde{\delta}_k^{-1}F_k(\tilde{\delta}_k D^{2s}(v_k(x+h)-v_k(x))+D^{2s}u_k(r_kx))\\
        &=\frac{1}{\theta_k^{1-\gamma}(r_k)}\left[\left(v_k(x+h)\right)_+^\gamma-\left(v_k(x+h)\right)_-^\gamma\right],
    \end{split}
\end{equation}
where $v_k$ is defined by \eqref{definitionvk}, and $\tilde{\delta}_k:=\theta_k(r_k)r_k^{\frac{\gamma}{1-\gamma}}$. Furthermore, as noted in the proof of \Cref{main result 1}, for any $R\ge1$, $v_k$ converges, up to a subsequence, to some $v_0$ in $C^{2s+\gamma}_{\loc}(B_R)$, as $k\to\infty$, and $v_0$ satisfies \eqref{limitseminormestimate2}. Since
$$
\theta_k(r_k)r_k^{\frac{\gamma}{1-\gamma}} \longrightarrow 0,
$$
as $k\to\infty$, using \Cref{convergence}, and passing to the limit in \eqref{fullynonlinearequation}, we get
\begin{equation}\label{limitequation}
    DF_0(0)D^{2s}v_0=0.
\end{equation}
On the other hand, since $DF_0(0)D^{2s}$ is a constant coefficient linear operator, then up to an affine change of variables, from \eqref{limitequation}, we conclude
$$
(-\Delta)^sv_0=0\quad \textrm{in} \ \R^n.
$$
Therefore, as in the proof of \Cref{main result 1}, \Cref{SeminormGrowth} implies $v_0\equiv0$, which contradicts \eqref{limitseminormestimate2}.
\end{proof}

\begin{remark}\label{another remark on one-phase problem}
By passing to the limit as $s\nearrow 1$, we extend the regularity result of \cite{T16} to the two-phase setting, also covering the case of fully nonlinear elliptic operators.
\end{remark}

\bigskip

\noindent {\small \textbf{Acknowledgments.} DP is partially supported by CNPq (grant 305680/2022-6) and CNPq/MCTI 10/2023 (grant 420014/2023-3). RT is supported by the King Abdullah University of Science and Technology (KAUST). JMU is partially supported by the King Abdullah University of Science and Technology (KAUST) and UID/00324 - Centre for Mathematics of the University of Coimbra. We thank the anonymous referee for the valuable suggestions that helped improve the manuscript.}

\end{document}